\documentclass[12pt]{article}

\usepackage{amsmath,amssymb,mathrsfs,wasysym}
\usepackage{amsthm}
\usepackage{enumerate}
\usepackage{graphicx}
\date{}
\newtheorem{theorem}{Theorem}
\newtheorem{prop}{Proposition}
\newtheorem{lemma}{Lemma}
\newtheorem{rem}{Remark}
\newtheorem{exmp}{Example}

\begin{document}
\author{\bf Edyta Bartnicka, Andrzej Matra{\'s}}

  \title{Tops of graphs of non-degenerate linear codes}
\maketitle
\begin{abstract}
Let $\Gamma_k(V)$ be the Grassmann graph whose vertex set ${\mathcal G}_{k}(V)$ is formed by all $k$-dimensional subspaces of an $n$-dimensional vector space $V$ over the finite field $F_q$ consisting of $q$ elements. We discuss its subgraph $\Gamma(n,k)_q$  with the vertex set ${\mathcal C}(n,k)_q$ consisting of all non-degenerate linear $[n, k]_q$ codes.
We study maximal cliques $\langle U]^{c}_{k}$ of $\Gamma(n,k)_q$, which are intersections of tops of $\Gamma_k(V)$ with ${\mathcal C}(n,k)_q$. We show when they are contained in a line of ${\mathcal G}_{k}(V)$ and then we  prove that $\langle U]^{c}_{k}$ is a maximal clique of $\Gamma(n,k)_q$ when it is not contained in a line of ${\mathcal G}_{k}(V)$. Furthermore, we show that the automorphism group of the set of such maximal cliques is isomorphic with the automorphism group of $\Gamma(n,k+1)_{q}$.

\end{abstract}

\paragraph{\rm\bf Keywords}
Linear code, Grassmann graph, Maximal cliques of Grassmann graph, Tops of Grassmann graph

\paragraph{\rm\bf Mathematics Subject Classification:}
51E22, 94B27

\section{Introduction}
\label{}
We start from   Grassmann graphs, which are important for many reasons, especially they are classical examples of distance regular graphs \cite{brouwer}.
They
have  been discussed in a number of books and papers (see, e.g.
\cite{Pankov2010, Pankov2015}).

The Grassmann graph  $\Gamma_k(V)$ whose vertex set ${\mathcal G}_{k}(V)$ is formed by\linebreak $k$-dimensional subspaces of an $n$-dimensional vector space $V$ over the $q$-element field $F_q$ can be considered as the graph of all linear  codes $[n, k]_q$. Following the language of coding theory, two distinct codes are connected by an edge if they have the maximal possible number of common codewords. Since the Grassmann graph is complete for $k=1$ and $k=n-1$, we deal with the case  $1<k<n-1$.

This paper focuses on the restriction of the Grassmann graph to the set ${\mathcal C}(n,k)_q$ of all  linear codes $[n,k]_{q}$ whose generator matrices do not contain zero columns, i.e., every coordinate functional does not vanish on these codes.  Such codes are called non-degenerate. The subgraph of the Grassmann graph with the set of vertices ${\mathcal C}(n,k)_q$, denoted by $\Gamma(n, k)_q$,  is investigated, e.g., in \cite{KP, KP2, Pankov2023}. The research of linear codes by a reduction to the non-degenerate case is a natural approach, see \cite{TVN}.
The subgraphs of the Grassmann graph formed by various types of linear
codes are considered in \cite{CGK, CG, KP3}. Maximal cliques in the graph of  simplex codes have already been discussed in \cite{KP4}.

Maximal cliques  of $\Gamma(n,k)_q$ are intersections of maximal cliques of $\Gamma_k(V)$ with the set ${\mathcal C}(n,k)_q$ and there are two types of maximal cliques in $\Gamma_k(V)$. These are stars which consist of all $k$-dimensional subspaces of $V$ containing a fixed $(k-1)$-dimensional subspace, and tops which consist of all\linebreak $k$-dimensional subspaces of $V$ contained in a fixed $(k+1)$-dimensional subspace. The article of Kwiatkowski and Pankov \cite{KP2} completely characterize  intersections of stars of $\Gamma_k(V)$ with the set ${\mathcal C}(n,k)_q$.  Here we complete the findings of \cite{KP2} and provide the thorough description of maximal cliques of $\Gamma[n, k]_q$ which are intersections of tops of $\Gamma_k(V)$ with the set ${\mathcal C}(n,k)_q$.

The starting point of our investigation is introducing the subspace $W$ of $(k+1)$-dimensional vector space $V^{k+1}$ over the finite field $F_q$. This subspace is generated by all vectors which are not scalar multiples of columns of the generator matrix for a fixed non-degenerate linear code $[n,k+1]_{q}$.  Next, we will show that in the case of $\dim W\leqslant2$ the non-degenerate part of the top of $\Gamma_k(V)$ is contained in a line of ${\mathcal G}_{k}(V)$ (Theorem \ref{line}). Then we will prove that the non-degenerate part of the top of $\Gamma_k(V)$ is a maximal clique of $\Gamma(n, k)_q$ when it is not contained in a line of ${\mathcal G}_{k}(V)$ (Theorem\ref{top}).
 Examples of non-degenerate parts of tops of $\Gamma_k(V)$
 are also given.
Finally, we will show that the automorphism group  of the set of such maximal cliques can be identified with the automorphism group of the graph of non-degenerate linear codes $[n,k+1]_{q}$ and we will discuss the action of this group on the set of maximal cliques. The characterization of automorphism group of linear codes and the linear codes equivalence problem 
have been  investigated  by many authors (see, e.g. \cite{Bal, Ba, HH}). New algorithms for equivalence of linear codes are constantly finding (see \cite{Bo2}).

\section{Graph of non-degenerate linear codes as a subgraph of Grassmann graph}
Let  $V$ be an  $n$-dimensional vector space  over the finite field $F_q$ consisting of $q$ elements and let ${\mathcal G}_{k}(V)$ be the Grassmannian consisting of all\linebreak $k$-dimensional subspaces of $V$. Two such subspaces are called {\it adjacent} if their intersection is $(k-1)$-dimensional.
The {\it Grassmann graph} $\Gamma_{k}(V)$ is the simple and connected graph whose vertex set is ${\mathcal G}_{k}(V)$. Its vertices are joined by an edge if, and
only if, they are adjacent.
In the cases when $k=1,\ k=n-1$, the Grassmann graph is complete.

Let $S\subset U\subset V$ and $\dim S<k<\dim U$. The set $K$ of all $k$-dimensional subspaces  of $V$ such that $S\subset K\subset U$ is denoted by $[S,U]_{k}$. We write  $\langle U]_{k}$ and $[S\rangle_{k}$ in the  cases when $S=0$ or $U=V$, respectively. If $$\dim S=k-1\;\mbox{ and }\;\dim U=k+1,$$ then $[S,U]_{k}$ is called a {\it line}.

One of the basic concepts of graph theory is that of a clique. A clique in
an undirected graph is a subset of the vertices such that any two distinct
vertices from this subset comprise an edge, i.e., the subgraph induced by these vertices
is complete.  A clique   is said to be {\it maximal} if it is not
properly contained in any clique.
From now on we suppose that $1<k<n-1$, then any  maximal clique of $\Gamma_{k}(V)$ is one of the following forms:
\begin{enumerate}
\item[$\bullet$] the {\it star}  $[S\rangle_{k}$ consisting of all $k$-dimensional subspaces containing a fixed $(k-1)$-dimensional subspace $S$, 
\item[$\bullet$] the {\it top}  $\langle U]_{k}$ consisting of all $k$-dimensional subspaces contained in a fixed $(k+1)$-dimensional subspace $U$. 
\end{enumerate}

Let us put $V=\underbrace{F_{q}\times\dots \times F_{q}}_{n}$. Recall that $V$ contains precisely
$$\genfrac{[}{]}{0pt}{}{n}{k}_q=
\frac{(q^n-1)(q^{n-1}-1)\cdots(q^{n-k+1}-1)}{(q-1)(q^2-1)\cdots(q^k-1)}$$
$k$-dimensional subspaces. Hence the number of $1$-dimensional subspaces of $V$, which is also the number of hyperplanes in $V$, is equal
$$\genfrac{[}{]}{0pt}{}{n}{1}_q=\genfrac{[}{]}{0pt}{}{n}{n-1}_q=
[n]_{q}=\frac{q^n-1}{q-1}.$$
Furthermore, there are exactly $[n-k]_q$ $(k+1)$-dimensional subspaces of $V$ containing a fixed $k$-dimensional subspace of $V$.

\noindent
The number of elements of $\langle U]_{k}$ is equal $[k+1]_{q}$.

The intersection of two distinct maximal cliques of the  form $\langle U]_{k}$ is either empty or it contains exactly one vertex.
The latter case goes ahead precisely when the corresponding  $(k+1)$-dimensional subspaces are adjacent.
The intersection of $[S\rangle_{k}$ and $\langle U]_{k}$ is non-empty exactly if
the corresponding\linebreak $(k-1)$-dimensional and $(k+1)$-dimensional subspaces are incident. This non-empty intersection is called a line of ${\mathcal G}_k(V)$ and any line of ${\mathcal G}_{k}(V)$ contains precisely $q+1$ elements.

The standard basis of $V$ consists of the vectors
$$e_{1}=(1,0,\dots,0),\dots,e_{n}=(0,\dots,0,1).$$
Let $C_i$ stands for the coordinate hyperplane of V spanned by all the vectors $e_j$ with $j\neq i$, i.e., the kernel of the $i$-th coordinate functional $(v_{1},\dots,v_{n})\to v_{i}$.

A {\it linear code} $[n,k]_{q}$ is an element of ${\mathcal G}_{k}(V)$. We consider only the case of {\it non-degenerate} linear codes. These are linear codes not contained in any coordinate hyperplane $C_{i}$. One can equivalently say, that $[n,k]_{q}$ is non-degenerate if
the restriction of every coordinate functional to $[n,k]_{q}$ is non-zero, i.e., its generator matrix does not contain zero columns.
${\mathcal C}(n,k)_q$ stands for the set of all non-degenerate linear codes $[n,k]_{q}$ and is equal to\linebreak ${\mathcal G}_{k}(V)\setminus \left(\bigcup_{i=1}^{n}{\mathcal G}_{k}(C_{i}) \right).$

Write $\Gamma(n, k)_q$  for the restriction of the  Grassmann graph to the set of all non-degenerate linear codes. This is the simple and connected graph \cite[Proposition 1]{KP} whose set of vertices is $C(n,k)_{q}$.
So, any clique of $\Gamma(n,k)_{q}$ is a  clique of $\Gamma_{k}(V)$ and  any maximal clique of $\Gamma(n,k)_{q}$ is
 the intersection of $\Gamma(n,k)_{q}$ with a maximal clique of $\Gamma_{k}(V)$.  We denote $[S\rangle_{k}\cap{\mathcal C}(n,k)_q$ and $\langle U]_{k}\cap{\mathcal C}(n,k)_q$ by $[S\rangle^{c}_{k}$ and $\langle U]^{c}_{k}$, respectively.
They both
can be empty or  non-maximal cliques of $\Gamma(n,k)_{q}$.

In the last section will be needed the concept of monomial semilinear automorphism of $V$. Recall that a bijection $f:V\rightarrow V$ is a {\it semilinear automorphism} if there exists a field automorphism $\sigma:F_q\rightarrow F_q$ such that for all $x, y\in V$ and $a\in F_q$ it holds that
\begin{enumerate}
  \item $f(x+y)=f(x)+f(y),$
  \item $f(ax)=\sigma(a)f(x).$
\end{enumerate}
Let $\delta$ be a permutation on the set $\{1, 2, \dots, n\}$. A semilinear automorphism $f$ of $V$ is {\it monomial} if $f(e_i)=e_{\delta_i}$. Two linear codes $[n,k]_{q}$ are called {\it semilinearly equivalent} if there exists a monomial semilinear automorphism of $V$ between them.

\section{Tops of graphs of non-degenerate linear\linebreak codes}

In this article, we investigate sets $\langle U]^{c}_{k}$. To be more precise, we are interested in those sets $\langle U]^{c}_{k}$ which are maximal cliques of $\Gamma(n,k)_{q}$ and we call them tops of $\Gamma(n,k)_{q}$.
It is  known that  any maximal clique of $\Gamma(n,k)_{q}$ is not contained in a line of ${\mathcal G}_{k}(V)$.
It is obvious that if a $(k+1)$-dimensional subspace $U$ of $V$ does not belong to ${\mathcal C}(n,k+1)_q$, then $\langle U]^{c}_{k}$ is empty. That is why throughout this paper we shall only study $U\in {\mathcal C}(n,k+1)_q$. A generator matrix of $U$, i.e., a matrix whose rows $v_{1},\dots,v_{k+1}$  form a basis for $U$, will be denoted by $M$.

The following lemma will be important in our consideration. For better understanding we prove it newly in a simpler way.
\begin{lemma}{\rm \cite{KP2}}
\label{lem.prop}
Intersections $U\cap C_{i}$ and $U\cap C_{j}$ coincide if, and only if, the $i$-th column of $M$ is a scalar multiple of the $j$-th column of $M$.
\end{lemma}
\begin{proof}
Let $v_{1},\dots,v_{k+1}$ be the rows of the matrix $M$ with $v_l=(v_l^{1},\dots,v_l^{n})$ for $l=1, 2, \dots, k+1$.
Then $i$-th and $j$-th coordinates of any linear combination of vectors $v_l$ are of the form $\sum_{l=1}^{k+1}a_lv_l^i$ or $\sum_{l=1}^{k+1}a_lv_l^j$, respectively.\linebreak
$U\cap C_{i}$ coincides with $U\cap C_{j}$ if, and only if, sets of solutions  of equations $\sum_{l=1}^{k+1}a_lv_l^i=0$ and $\sum_{l=1}^{k+1}a_lv_l^j=0$ (with unknowns $a_l, \ l=1, 2, \dots, k+1$) are equal.
The fact that the dimension of the space of solutions of the linear system $$\begin{cases}
\sum_{l=1}^{k+1}a_lv_l^i=0\\\sum_{l=1}^{k+1}a_lv_l^j=0\end{cases}$$ is equal $k$ implies $$rank\left[
\begin{array}{cccc}
v_1^i&v_2^i&\dots&v_{k+1}^i\\
v_1^j&v_2^j&\dots&v_{k+1}^j\\
\end{array}
\right]=1,$$
which completes the proof.
\end{proof}

Lemma \ref{lem.prop} suggests introducing
the subset $W'$ of $V^{k+1}$ formed by all non-zero vectors $w=(w_{1},\dots,w_{k+1})$ whose coordinates are not proportional to coordinates of any column of $M$. The  subspace of $V^{k+1}$ generated by $W'$ is denoted by $W$. As it was done in \cite{KP2}  we define the subspace $C(w)$ of $U$ as the set of all
vectors $\sum^{k+1}_{i=1}a_{i}v_{i}$ with scalars $a_{1},\dots, a_{k+1}\in F_q$ satisfying the equality
$$\sum^{k+1}_{i=1}a_{i}w_{i}=0.$$
\begin{exmp}
  This example concerns the case when $\dim W=1$. Namely, let $$ M=
\left[\begin{array}{cccccccccccc}
1&0&0&1&1&0&1&1&1&0&1&1\\
0&1&0&1&0&1&1&2&0&1&1&2\\
0&0&1&0&1&1&1&0&2&2&2&1\\
\end{array}
\right]$$ generates $U\in{\mathcal C}(12,3)_3$.
Coordinates of columns of $M$ are mutually disproportional, hence $U\cap C_i\neq U\cap C_j$ for $i\neq j$, as follows from Lemma \ref{lem.prop}. Therefore, there are $12$ distinct $2$-dimensional subspaces $U\cap C_i$. On the other hand, there are precisely $[3]_3=13$ distinct $2$-dimensional subspace of $U$. Thus, there exists exactly one non-degenerate linear code $[12, 2]_3\subset U$. Its generating matrix  is of the form $$\left[\begin{array}{cccccccccccc}
1&1&0&2&1&1&2&0&1&1&2&0\\
1&0&1&1&2&1&2&1&0&2&0&2\\
\end{array}
\right].$$


\end{exmp}

The statement of the undermentioned lemma comes from \cite{KP2}. To be more precise, it was concluded by the proof of \cite[Lemma 3]{KP2}. Here we give a constructive proof of this lemma which will be needful  in further studies.
\begin{lemma}\label{ukc=cw}\cite{KP2}
 The set $\langle U]^{c}_{k}$ consists of all $C(w)$.
\end{lemma}\begin{proof}
Let $w$ be an element of $W'$. Since $w$ is a non-zero vector, so  there exists $i\in \{1, \dots, k+1\}$ such  that
$$a_i
=-(w_i)^{-1}\sum_{l\in \{1,2, \dots, k+1\}\backslash\{i\}}a_lw_l.$$
Hence $$\sum^{k+1}_{i=1}a_{i}v_{i}=
\sum_{l\in \{1,2, \dots, k+1\}\backslash\{i\}}a_l\left(v_l-(w_i)^{-1}w_lv_i\right).$$
Consequently, the set of vectors $\left\{v_l-(w_i)^{-1}w_lv_i;\ l\in \{1,2, \dots, k+1\}\backslash\{i\}\right\}$ spans $C(w)$ and it is linearly independent, so $\dim C(w)=k$.

\noindent
Additionally, $C(w)$ is not contained in any coordinate hyperplane $C_i$ for all $w\in W'$, otherwise we would obtain the following system of equations for some $w\in W'$:
$$\begin{cases}
\sum_{l=1}^{k+1}a_lv_l^i=0\\\sum_{l=1}^{k+1}a_lw_l=0\end{cases}$$ with unknowns $a_1, \dots, a_{k+1}$. The dimension of the space of its solutions is equal $(k+1)-2=k-1$ which contradicts $\dim C(w)=k$. We thus get $C(w)\in \langle U]^{c}_{k}$ for all $w\in W'$.

\noindent
Furthermore, any vector of $k$-dimensional subspace of $U$ is of the form\linebreak $\sum^{k+1}_{i=1}a_{i}v_{i}$, where $a_{1},\dots, a_{k+1}$ are elements of $F_q$ and there exists a vector $w=(w_1, \dots, w_{k+1})$ such that $\sum^{k+1}_{i=1}a_{i}w_{i}=0$ for any $a=(a_{1},\dots, a_{k+1})$.
Suppose now that there exist $b\in F_q$ and $s\in\{1, \dots, n\}$ such that $bw$ is the $s$-th column of $M$, i.e., $w=b^{-1}(v_1^s, v_2^s, \dots, v_{k+1}^s)$. This yields
$\sum_{i=1}^{k+1}a_iv_i^s=0$, what means that the considered  $k$-dimensional subspace of $U$
 is contained in the coordinate hyperplane $C_s$. Therefore, any $k$-dimensional subspace of $U$ not contained in any coordinate hyperplane $C_i$ is one of the  subspaces $C(w)$  defined above. In fine, $\langle U]^{c}_{k}$ consists of all $C(w)$, where  $w\in W'$.\end{proof}
\begin{prop}\label{dimCw}
 Let $W\neq \emptyset$ and let $W^\perp$ be the orthogonal complement of the subspace $W$ of $V^{k+1}$. Then
   $$\dim W^\perp=\dim\bigcap_{w\in W'}C(w).$$
\end{prop}
\begin{proof}
Let $a=(a_1, \dots, a_{k+1})\in W^\perp$, i.e., $\sum^{k+1}_{i=1}a_{i}w_{i}=0$ holds for all $w=(w_1, \dots, w_{k+1})\in W'$. This is equivalent to saying that $\sum^{k+1}_{i=1}a_{i}v_{i}\in \bigcap_{w\in W'}C(w)$, what yields the desired claim.
\end{proof}
   It is known that if $W^\perp$ is the orthogonal complement of the subspace $W$ of $V^{k+1}$, then
   $$\dim W+\dim W^\perp=k+1.$$
A crucial property of sets $\langle U]^{c}_{k}$ is stated in the following result:
\begin{theorem}\label{line}
The set $\langle U]^{c}_{k}$ is contained in a line of  ${\mathcal G}_{k}(V)$ if, and only if, $\dim W\leqslant2$.
\end{theorem}
\begin{proof}
$"\Rightarrow"$ \
 Assume that $\dim W>2$. Equivalently, $\dim W^\perp<k-1$, where $W^\perp$ is the orthogonal complement of the subspace $W$ of $V^{k+1}$. In  light of Proposition \ref{dimCw} $$\dim\bigcap_{w\in W'}C(w)<k-1$$ and hence $\langle U]^{c}_{k}$ is not contained in a line of ${\mathcal G}_{k}(V)$.

 $"\Leftarrow"$ \ If $\dim W=0$, $\langle U]^{c}_{k}$ is empty.  Suppose now that $\dim W=1$ or $\dim W=2$. Equivalently, $\dim W^\perp=k$ or $\dim W^\perp=k-1$, and by Proposition \ref{dimCw} $\dim\bigcap_{w\in W'}C(w)=k$ or $\dim\bigcap_{w\in W'}C(w)=k-1$. It is obvious that in both cases there exists $(k-1)$-dimensional subspace $S\subseteq\bigcap_{w\in W'}C(w)$. Thereby $\langle U]^{c}_{k}$ is contained in a line of  $[S, U]_{k}$.
\end{proof}

\begin{exmp}
  Consider the non-degenerate linear code $U\in{\mathcal C}(360,6)_3$ 
  with the set $W'$ consists of vectors
  \begin{align*}
  & w_1=(1,2,1,2,1,2),\\ & w_2=(0,1,0,1,0,1),\\ & w_3=(1,0,1,0,1,0),\\ & w_4=(1,1,1,1,1,1)\end{align*} and their scalar multiples $w_5=2w_1,\ w_6=2w_2,\ w_7=2w_3,\  w_8=2w_4$.\linebreak
 Then $\dim W=\dim\left\langle w_1,w_2,w_3,w_4,w_5,w_6,w_7,w_8\right\rangle=2$, so $\dim W^\perp=4$\linebreak  and
   $W^\perp=\left\langle (1,0, 0, 0, 2, 0),(0,1,0,0,0,2),(0,0,1,0,2,0),(0,0,0,1,0,2)\right\rangle$.\linebreak
  According to Theorem \ref{line}, the set $\langle U]^{c}_{5}$ is contained in a line of  ${\mathcal G}_{5}(V)$. We will show now that $\langle U]^{c}_{5}$ is a line of  ${\mathcal G}_{5}(V)$.

  \noindent
  Based on the proof of Proposition \ref{dimCw} and Lemma \ref{lem.prop} we get \begin{align*}
  & C(w_1)=C(w_5)=\left\langle v_1+2v_5, v_2+2v_6, v_3+2v_5, v_4+2v_6, v_1+v_2\right\rangle,\\ & C(w_2)=C(w_6)=\left\langle v_1+2v_5, v_2+2v_6, v_3+2v_5, v_4+2v_6, v_1\right\rangle,\\ & C(w_3)=C(w_7)=\left\langle v_1+2v_5, v_2+2v_6, v_3+2v_5, v_4+2v_6, v_2\right\rangle,\\ &C(w_4)=C(w_8)=\left\langle v_1+2v_5, v_2+2v_6, v_3+2v_5, v_4+2v_6, v_1+2v_2\right\rangle\end{align*} and $$\bigcap_{i=1}^4C(w_i)=\left\langle v_1+2v_5, v_2+2v_6, v_3+2v_5, v_4+2v_6\right\rangle.$$

  Therefore, $\langle U]^{c}_{5}$ is a line $\left[\bigcap_{i=1}^4C(w_i), U\right]_5$ of ${\mathcal G}_{5}(V)$.

\end{exmp}

\begin{theorem}\label{top}
  The set $\langle U]^{c}_{k}$ is a top of $\Gamma(n,k)_{q}$ if, and only if, it
  is not contained in a line of  ${\mathcal G}_{k}(V)$.
\end{theorem}
\begin{proof}
  $"\Rightarrow"$\ This is known (see \cite{KP2}).

  \noindent
  $"\Leftarrow"$\ Suppose that $\langle U]^{c}_{k}$
  is not contained in a line of  ${\mathcal G}_{k}(V)$. In view of Theorem \ref{line}, this is equivalent to saying that $\dim W>2$. Another way to state this is to say $\dim W^\perp<k-1$ and by Proposition \ref{dimCw} $\dim\bigcap_{w\in W'}C(w)<k-1$, what means that $\langle U]^{c}_{k}$ is not contained in any star of $\Gamma(n,k)_{q}$.
  Furthermore, the inequality $\dim W>2$ leads to $|\langle U]^{c}_{k}|>2$ (see Lemma \ref{ukc=cw}). Thus $\langle U]^{c}_{k}$ is not properly contained in any top of  $\Gamma(n,k)_{q}$. Therefore, $\langle U]^{c}_{k}$ is a top of $\Gamma(n,k)_{q}$.
\end{proof}
\begin{exmp}
  Choose the non-degenerate linear code $U\in{\mathcal C}(58,6)_2$  with the set $W'$ consists of vectors
  \begin{align*}
  & w_1=(1,0,1,0,1,0),\\ & w_2=(0,1,0,1,0,1),\\ & w_3=(1,1,1,0,0,0),\\ & w_4=(1,1,1,1,1,1), \\ & w_5=(0,1,0,0,1,0).\end{align*} Then $\dim W=\dim\left\langle w_1,w_2,w_3,w_4,w_5\right\rangle=3$. According to Theorem \ref{line} $\langle U]^{c}_{5}$ is not contained in any line of  ${\mathcal G}_{5}(V)$. By Theorem \ref{top} we get immediately that the set $\langle U]^{c}_{5}$ is a top of $\Gamma(58,5)_{2}$. Clearly,  $$W^\perp=\left\langle (1,0,1,0,0,0),(0,1,1,0,1,1),(0,0,0,1,0,1)\right\rangle,$$
   and then by Proposition \ref{dimCw} $\dim W^\perp=\dim \bigcap_{i=1}^5C(w_i)=3.$
   Moreover, based on the proof of Proposition \ref{dimCw} we get that \begin{align*}
  & C(w_1)=\left\langle v_1+v_3, v_2+v_3+v_5+v_6,  v_4+v_6, v_2, v_4\right\rangle,\\ & C(w_2)=\left\langle v_1+v_3, v_2+v_3+v_5+v_6,  v_4+v_6, v_2+v_4, v_5\right\rangle,\\ & C(w_3)=\left\langle v_1+v_3, v_2+v_3+v_5+v_6,  v_4+v_6, v_4, v_5\right\rangle,\\ &C(w_4)=\left\langle v_1+v_3, v_2+v_3+v_5+v_6,  v_4+v_6, v_2+v_4, v_4+v_5\right\rangle\\ &C(w_5)=\left\langle v_1+v_3, v_2+v_3+v_5+v_6,  v_4+v_6, v_4, v_2+v_5\right\rangle\end{align*}
 make up the top $\langle U]^{c}_{5}$ of $\Gamma(58,5)_{2}$.
\end{exmp}
Notice that \cite[Corollary 1]{KP2} is now easy consequence of our results. The number of all scalar multiples of two linearly independent vectors of $W$ is $q+1$, so the inequality $[k+1]_q>n+q+1$ means that there are more then two linearly independent vectors in $W'$. According to Theorems \ref{line} and \ref{top}, this is equivalent to saying that,  $\langle U]^{c}_{k}$ is a top of $\Gamma(n,k)_{q}$.

By the above consideration we classify below  non-maximal cliques $\langle U]^{c}_{k}$ according to a number of lines of ${\mathcal G}_{k}(V)$ containing them.
\begin{prop}
  There are exactly three types of non-maximal cliques $\langle U]^{c}_{k}$ for any $U\in {\mathcal C}(n,k+1)_q$. These are:
  \begin{enumerate}
    \item \label{1} $\langle U]^{c}_{k}=\emptyset$;
    \item \label{2} The set $\langle U]^{c}_{k}$ contained in precisely $[k]_q[n-k]_q$ lines of  ${\mathcal G}_{k}(V)$;
    \item \label{3} The set $\langle U]^{c}_{k}$ contained in precisely one line of  ${\mathcal G}_{k}(V)$ that is $$[\bigcap_{w\in W'}C(w), U]_k.$$
  \end{enumerate}
\end{prop}
\begin{proof}
According to Theorems \ref{line} and \ref{top}, we see immediately that $\langle U]^{c}_{k}$ is a non-maximal clique exactly if $\dim \langle W\rangle\leqslant 2$. Suppose first that $\dim \langle W\rangle=0$. Equivalently, $\langle U]^{c}_{k}=\emptyset$ and is contained in every line of ${\mathcal G}_{k}(V)$. Let now $\dim \langle W\rangle=1$, i.e., there exists exactly one subspace $C(w)\in \langle U]^{c}_{k}$. Of course, $C(w)$ contains $[k]_q$ $(k-1)$-dimensional subspaces $S$ and there exist  $[n-k]_q$ $(k+1)$-dimensional subspace $U'$ of $V$ containing $C(w)$, thus $\langle U]^{c}_{k}$ is contained in exactly $[k]_q[n-k]_q$ lines $[S, U']_k$. The last case is $\dim \langle W\rangle=2$. By Proposition \ref{dimCw} this is equivalent to saying that  $\dim \langle \bigcap_{w\in W'}C(w)\rangle=k-1$ and $\bigcap_{w\in W'}C(w)$ is the only $(k-1)$-dimensional subspace of $V$  contained in any element of $\langle U]^{c}_{k}$. Furthermore, $\dim \langle W\rangle=2$ indicates $\langle U]^{c}_{k}\subseteq\langle U']^{c}_{k}$ if, and only if, $U=U'$. Consequently, $[\bigcap_{w\in W'}C(w), U]_k$ is the only line of ${\mathcal G}_{k}(V)$ containing $\langle U]^{c}_{k}$.
\end{proof}

\section{Automorphisms of the set of tops}

The last section focuses on automorphisms of the set $U(n, k)_q$ of tops $\langle U]^{c}_{k}$, where $k\in\{2, \dots, n-2\}$ is fixed. The set of all automorphisms  of  $U(n, k)_q$ forms a group which is denoted by $Aut \left(U(n, k)_q\right)$.
Since tops $\langle U]^{c}_{k}$ are uniquely determined by $k+1$-dimensional subspaces $U\in{\mathcal C}(n,k+1)_q$ and automorphisms of ${\mathcal C}(n,k+1)_q$ transform tops into tops, every automorphism of the set  $U(n, k)_q$ is an automorphism of ${\mathcal C}(n,k+1)_q$. They map non-degenerate linear codes to non-degenerate linear codes and so we get the following lemma.
\begin{lemma}\label{aut}
The automorphism group $Aut \left(U(n, k)_q\right)$ of the set of tops $\langle U]^{c}_{k}$, where $k\in\{2, \dots, n-2\}$ is fixed, can be identified with the automorphism group $Aut \left(\Gamma(n,k+1)_{q}\right)$  of the graph of non-degenerate linear  codes $[n, k+1]_q$.
\end{lemma}
In the case of $k=n-2>1$ any graph $\Gamma(n,k+1)_{q}$ is complete and any bijective transformation of the vertex set of this graph is its automorphism. The order of the automorphism group is $|\Gamma(n,k+1)_{q}|!$ and its action on the set $U(n, k)_q$ has exactly one orbit.

From now on we will consider only $k\in\{2, \dots, n-3\}$.

 Any monomial semilinear automorphism of $V$ induces an automorphism of $\Gamma(n,k+1)_{q}$ and any automorphism of $\Gamma(n,k+1)_{q}$  is induced by a monomial semilinear automorphism of $V$ \cite[Theorem 2]{KP2}. Therefore, in light of Lemma \ref{aut} an orbit of an element of the set $U(n, k)_q$ under the action of its automorphism group is made  by tops determined by semilinearly equivalent non-degenerate linear codes $[n, k+1]_q$.
 If $q$ is a prime  number, then the automorphism group $Aut(F_q)$ of $F_q$ is trivial 
  and hence the automorphism group $Aut \left(U(n, k)_q\right)$ is equal to the linear automorphism group $Aut_L \left(U(n, k)_q\right)$. Otherwise, $Aut_L \left(U(n, k)_q\right)$ is the proper subgroup of $Aut \left(U(n, k)_q\right)$ and $|Aut \left(U(n, k)_q\right)|=|Aut_L \left(U(n, k)_q\right)|\cdot|Aut(F_q)|$.
\begin{prop}\label{orbit}
 Let  $k\in\{2, \dots, n-3\}$ be fixed.  An orbit of an element of the set $U(n, k)_q$ under the action of $Aut_L \left(U(n, k)_q\right)$ is made  by tops whose generator matrices consist of the same columns up to proportionality.
\end{prop}
\begin{proof}
  Suppose that tops $\langle U]^{c}_{k}$, $\langle U']^{c}_{k}$ are in the same orbit. Equivalently, there exists an automorphism of the set of tops  transferring one of them to the other. By the above this automorphism is induced by a monomial linear automorphism of $V$. It can be represented by a monomial matrix $A$ of order $n$ with
entries from $F_q$. Let now  $M, M'$ be generator matrices of $\langle U]^{c}_{k}$, $\langle U']^{c}_{k}$, respectively. The result of multiplication $MA=M'$ is a matrix, whose columns are the same as columns of $M$ up to proportionality.
\end{proof}

It leads to the simple conclusion that the total number of linear automorphisms of the set $U(n, k)_q$ for $1<k<n-2$ is $n!(q-1)^n$.

\begin{rem}
Suppose that two columns in the generator  matrix $M$ of $U\in{\mathcal C}(n,k+1)_q$, where $1<k<n-2$, are equivalent if they are proportional  and denote by $k_1, k_2,\dots, k_s$ numbers of elements in each equivalence class. Then
there are exactly \begin{itemize}
                        \item $k_1!k_2! \dots k_s!$ linear automorphisms of the set $U(n, k)_q$, which are automorphisms of $M$;
                        \item $\frac{n!}{k_1!k_2!\dots k_s!}(q-1)^n$ matrices $(k+1)\times n$ without zero columns, which consists of the same columns up to proportionality.
                      \end{itemize}\end{rem}

Two matrices $M,\ M'$ generate the same $U\in{\mathcal C}(n,k+1)_q$ if there exists an invertible matrix $S$ of order $k+1$ over $F_q$  such that $SM=M'$. So all automorphisms represented by matrices $A$ which satisfy $SM=MA$ for some $S$ make up the stabilizer of $\langle U]^{c}_{k}$ in $Aut_L \left(U(n, k)_q\right)$.

\begin{exmp}
  Let $q>2$. Consider the top $\langle U]^{c}_{2}$ whose generator matrix is $$M=\left[\begin{matrix}
                                         1 & 0 & 0 & 1 & 0 \\
                                         0 & 1 & 0 & 1 & 1 \\
                                         0 & 0 & 1 & 0 & 1
                                       \end{matrix}\right].$$
The set of all matrices generating $U$ consists of elements $$SM=\left[\begin{matrix}
                                         s_{11} & s_{12} & s_{13} & s_{11}+s_{12} & s_{12}+s_{13} \\
                                         s_{21} & s_{22} & s_{23} & s_{21}+s_{22} & s_{22}+s_{23} \\
                                         s_{31} & s_{32} & s_{33} & s_{31}+s_{32} & s_{32}+s_{33}
                                       \end{matrix}\right],$$ where $S=\left[\begin{matrix}
                                         s_{11} & s_{12} & s_{13}  \\
                                         s_{21} & s_{22} & s_{23}  \\
                                         s_{31} & s_{32} & s_{33}
                                       \end{matrix}\right]$ is any invertible matrix with entries from $F_q$.
The only matrices $A$, which satisfy $SM=MA$ for some $S$ are of the form:
\begin{align*}
                                                                                                                    & \left[\begin{matrix}
                                                                                                                    a & 0 & 0 & 0 & 0 \\
                                                                                                                    0 & a & 0 & 0 & 0 \\
                                                                                                                    0 & 0 & a & 0 & 0 \\
                                                                                                                    0 & 0 & 0 & a & 0 \\
                                                                                                                    0 & 0 & 0 & 0 & a
                                                                                                                  \end{matrix}\right], \left[\begin{matrix}
                                                                                                                    a & 0 & 0 & 0 & 0 \\
                                                                                                                    0 & a & 0 & 0 & 0 \\
                                                                                                                    0 & 0 & 0 & 0 & -a \\
                                                                                                                    0 & 0 & 0 & a & 0\\
                                                                                                                    0 & 0 & -a & 0 & 0
                                                                                                                  \end{matrix}\right],
\left[\begin{matrix}
                                                                                                                    0 & 0 & a & 0 & 0 \\
                                                                                                                    0 & a & 0 & 0 & 0 \\
                                                                                                                    a & 0 & 0 & 0 & 0 \\
                                                                                                                    0 & 0 & 0 & 0 & a \\
                                                                                                                    0 & 0 & 0 & a & 0
                                                                                                                  \end{matrix}\right], \\
                                                                                                                    & \left[\begin{matrix}
                                                                                                                    0 & 0 & a & 0 & 0 \\
                                                                                                                    0 & a & 0 & 0 & 0 \\
                                                                                                                    0 & 0 & 0 & -a & 0 \\
                                                                                                                    0 & 0 & 0 & 0 & a \\
                                                                                                                    -a & 0 & 0 & 0 & 0
                                                                                                                  \end{matrix}\right],
\left[\begin{matrix}
                                                                                                                    0 & 0 & 0 & a & 0 \\
                                                                                                                    0 & -a & 0 & 0 & 0 \\
                                                                                                                    0 & 0 & -a & 0 & 0 \\
                                                                                                                    a & 0 & 0 &  & 0 \\
                                                                                                                    0 & 0 & 0 & 0 & -a
                                                                                                                  \end{matrix}\right],
 \left[\begin{matrix}
                                                                                                                    0 & 0 & 0 & a & 0 \\
                                                                                                                    0 & -a & 0 & 0 & 0 \\
                                                                                                                    0 & 0 & 0 & 0 & a \\
                                                                                                                    a & 0 & 0 & 0 & 0 \\
                                                                                                                    0 & 0 & a & 0 & 0
                                                                                                                  \end{matrix}\right], \\
                                                                                                                    & \left[\begin{matrix}
                                                                                                                    0 & 0 & 0 & 0 & a \\
                                                                                                                    0 & -a & 0 & 0 & 0 \\
                                                                                                                    -a & 0 & 0 & 0 & 0 \\
                                                                                                                    0 & 0 & a & 0 & 0 \\
                                                                                                                    0 & 0 & 0 & -a & 0
                                                                                                                  \end{matrix}\right],
 \left[\begin{matrix}
                                                                                                                    0 & 0 & 0 & 0 & a \\
                                                                                                                    0 & -a & 0 & 0 & 0 \\
                                                                                                                    0 & 0 & 0 & a & 0 \\
                                                                                                                    0 & 0 & a & 0 & 0 \\
                                                                                                                    a & 0 & 0 & 0 & 0
                                                                                                                  \end{matrix}\right],
                                                                                                                 \end{align*}
                                                                                                                 where $a\in F_q\setminus\{0\}$.
Thus the stabilizer of $\langle U]^{c}_{2}$ in $Aut_L \left(U(5, 2)_q\right)$  consists of $8(q-1)$ automorphisms represented by the above-listed matrices. Consequently,   the length of the orbit of $\langle U]^{c}_{2}$ is $15(q-1)^4$.

\end{exmp}

Any linear code can be associated with a specific binary matrix, such that their automorphism groups are isomorphic (see \cite{Bo}). Furthermore, it is also well known that any such matrix can be considered as a bipartite graph. It follows that the description of the stabilizer of $\langle U]^{c}_{k}$ in $Aut \left(U(n, k)_q\right)$ can be based on the automorphism group of respective bipartite graph.
In the case of $M$ such that elementary row operations preserve columns of $M$, the automorphism group of $U$, $Aut(U)$, is simultaneously the stabilizer of $\langle U]^{c}_{k}$ in $Aut \left(U(n, k)_q\right)$.

\begin{exmp}
Consider the linear code $U\in{\mathcal C}(n,k+1)_2$ whose generator matrix is  $$M=\left[\begin{array}{lccc|cccc|c|ccccr}
                                                                                                                    1 & 1 & &1& 0 & 0 & & 0 & &0&0& &0 \\
                                                                                                                    0 & 0 & &0& 1& 1 & &1& & 0 &0& & 0 \\
                                                                                                                    0 & 0 &\ldots&0& 0 & 0&\ldots&0&\ldots &\vdots & \vdots&\ldots&\vdots \\
                                                                                                                    \vdots&\vdots& &\vdots&\vdots&\vdots& &\vdots& &0&0& &0\\
                                                                                                                    0&0& &0&0&0& &0& &1&1& &1
                                                                                                                  \end{array}\right]$$
with $t$ identical columns in each equivalence class of proportionality.
It can be considered as the bipartite graph $G(A, B, E)$ which is depicted in Figure \ref{fig}. Elements of $A=\{r_1, \dots, r_{k+1}\}$ refer to rows of $M$, and elements of $B=\{c_{i1}, \dots, c_{it}; i=1, \dots, k+1\}$ refer to columns of $M$.
 \begin{figure}[ht]
\centering
\includegraphics[width=0.5\textwidth]{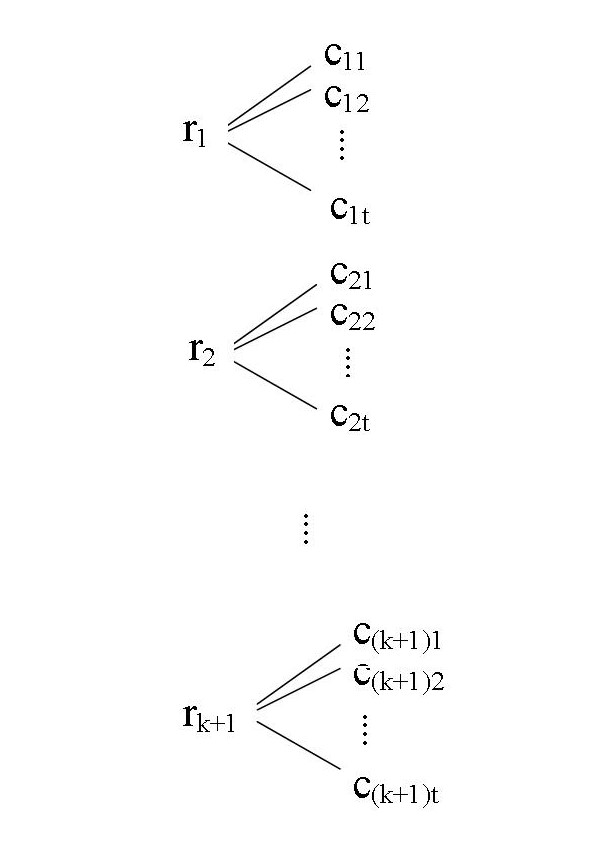}
\caption{Bipartite graph related to $U\in{\mathcal C}(n,k+1)_2$}\label{fig}
\end{figure}

\noindent
The total number of its automorphisms is $(t!)^{k+1}(k+1)!$. They form a group which is isomorphic with  $Aut(U)$ equal to the stabilizer of $\langle U]^{c}_{k}$ in $Aut \left(U(n, k)_2\right)$. $(t!)^{k+1}$ means the number of automorphisms of $M$ and $(k+1)!$ is the number of permutations of the rows of $M$ which are all elementary row operations preserving columns of $M$.
Automorphisms of $\langle U]^{c}_{k}$ are induced by automorphisms of $V$ and thus $\left|Aut \left(U(n, k)_2\right)\right|=n!$. An orbit of  $\langle U]^{c}_{k}$ under the action of $Aut\left(U(n, k)_2\right)$ consists of $\frac{n!}{(t!)^{k+1}(k+1)!}$ tops.

The set $W'$ consists of $2^{k+1}-(k+2)$ vectors with at least two non-zero coordinates.
  The group $Aut(U)$ acts on the set $\langle U]^{c}_{k}$. Moreover, the group $Aut(U)$ contains a normal subgroup $\texttt{T}\simeq \underbrace{S_t\times\dots \times S_t}_{k+1} $ which acts identically on the set $\langle U]^{c}_{k}$ and ${Aut(U)}/\texttt{T}\simeq S_{k+1}$.
\end{exmp}





\footnotesize Edyta Bartnicka\\
Institute of Information Technology,
Faculty of Applied Informatics
and Mathematics,
Warsaw University of Life Sciences - SGGW,
Nowoursynowska 166 St.,
02-787 Warsaw,
Poland\\
{\tt edyta\_bartnicka@sggw.edu.pl}

\vspace{0.3cm}

Andrzej Matra{\'s}\\
Faculty of Mathematics and Computer Science,
University of Warmia and Mazury in Olsztyn,
S{\l}oneczna 54 St.,
10-710 Olsztyn,
Poland\\
{\tt matras@uwm.edu.pl}\normalsize
\end{document}